\documentclass[smallextended]{svjour3}       
\usepackage{graphicx}
\usepackage[symbol]{footmisc}
\usepackage[misc]{ifsym}
\usepackage{hyperref}
\usepackage{amssymb}
\usepackage{latexsym,amscd}
\usepackage{color}
\usepackage{slashed}
  \newcommand{\tr}{\mathrm{Tr}}

  \newcommand{\beq}{\begin{equation}}
  \newcommand{\eeq}{\end{equation}}

\newcommand{\OM}{\Omega}

\newcommand{\DEL}{\Delta}

\newcommand{\la}{\lambda}

\newcommand{\RR}{\mathbb{R}}
\newcommand{\BR}{\mathbb{R}^{n}}
\newcommand{\CC}{\mathbb{C}}

\newcommand{\lka}{\langle}
\newcommand{\rka}{\rangle}

\renewcommand{\d}{\textrm{\rm d}}
\newcommand{\e}{\textrm{\rm e}}

\emergencystretch=\maxdimen
\begin{document}

\title{Some domination inequalities for spectral zeta kernels on closed Riemannian manifolds}



\author{Louis Omenyi\thanks{Corresponding author: Louis Omenyi, \\ Email: omenyi.louis@funai.edu.ng}$^{1, \href{https://orcid.org/0000-0002-8628-0298}{ID}}$         \and
        McSylvester  Omaba$^{2}$ 
}


\institute{$^{\textrm{\Letter}}$ Louis Omenyi \at
              Department of Mathematics/Computer Science/Statistics/Informatics,\\
              Alex Ekwueme Federal University, Ndufu-Alike, Nigeria \\
              \email{omenyi.louis@funai.edu.ng}           
           \and
           McSylvester Omaba \at
              Department of Mathematics, College of Science, University of Hafr Al Batin, 
 P. O Box 1803 Hafr Al Batin 31991, KSA
}

\date{Received: date / Accepted: date}

\maketitle

\begin{abstract} 

We first prove Kato's inequalities for the Laplacian and a Schr$\ddot{o}$dinger-type 
operator on smooth functions on closed Riemannian manifolds. We then apply the result to establish
some new domination inequalities for spectral zeta functions and their related spectral zeta kernels 
on $n$-dimensional unit spheres using Kato's inequalities and majorisation techniques. Our results 
are the  generalisations of Kato's comparison inequalities for Riemannian surfaces 
 to $n$-dimensional closed Riemannian manifolds.
\end{abstract}

\section{Introduction}
A  generalisation of the  Riemann zeta function is the spectral 
zeta function, which is  the function of interest in this paper. 
We denote  the Laplace-Beltrami operator simply called the  Laplacian in many literature 
 acting on smooth functions on a Riemannian manifold $(M,g)$
 by $\Delta _{g}.$  The Laplacian defined locally on closed and connected  
 Riemannian  $n$-dimensional Riemannian manifold $(M,g)$ by  
\begin{equation}\label{kat2}
\Delta _{g} = - ~ \mathrm{div}(\mathrm{grad}) = - \frac{1}{\sqrt{|g|}} 
\sum _{i,j} \frac{\partial}{\partial x^{i}} \big( \sqrt{|g|}
 g^{ij} \frac{\partial}{\partial x^{j}} \big)
\end{equation}
where $g^{ij}$ are the components of  the dual metric  on the cotangent
 bundle $T^{\ast}_{x} M$ is a non-negative operator on smooth functions.
 
The operator $\Delta _{g}$ extends to a self-adjoint operator on 
$L^{2} (M) \supset H^{2} (M) \rightarrow L^{2} (M) $ with compact resolvent. This implies that there
 exists an orthonormal basis $\psi _{k} \in L^{2}(M)$ consisting of eigenfunctions such that 
\begin{equation}\label{kat3}
\DEL _{g}  \psi_{k} = \la _{k} \psi_{k}
\end{equation}
where the  eigenvalues are listed with multiplicities.
  
It has the discrete spectrum $\{ \la_k \}_{k=1}^\infty$ listed with multiplicities. 
For details one may see for  example \cite{AG}, \cite{ICH}, \cite{SM} and 
 \cite{OM19} among many literature. The Laplacian $\Delta _{g}$ thus, has one-dimensional 
 null space consisting precisely of constant functions.  
 
Consequently, we define the spectral zeta function of the Laplacian on smooth functions 
on closed Riemannian manifolds by 
\begin{equation}\label{RZ00}
\zeta _{g} (s) = \sum _{k = 1} ^{\infty}
 \frac{1}{\la _{k} ^{s}}~;~  ~~~\Re (s) >\frac{n}{2}.
\end{equation}

Kato in \cite{Kato2} in effort to prove essential self-adjointness for 
Schr$\ddot{o}$dinger operators under very mild restrictions on the potential term  
introduced Kato's inequality and the Kato class of potentials 
that were  combined to give new insight in analysis and geometry of Riemannian manifolds. 
It computed explicit majorisation Kato's inequalities for the trace of the heat operator 
and the semigroup generated by Bochner Laplacian on forms. 

A characterisation of the generators of positive semigroups were done in \cite{Are}. 
It showed that the semigroup consists of positive operators  if and only if it satisfies 
the abstract version of the kato's inequality. It also characterised the domination of 
semigroups by an inequality for their generators. In continuation of the characterisation, 
Hess, et al \cite{HSU} constructed Kato's inequalities for several Laplacians on 
tangent bundles of some Riemannian manifolds and studied their semigroups dominations. 
 
B$\ddot{a}$r \cite{BAR2} used Kato's comparison principle for heat semi-groups to derive 
estimates for trace of the heat operator on surfaces with variable curvature. These 
estimates is from above for positively curved surfaces of genus $0$ and from below for genus 
greater or equal to $2.$ It was shown that the estimates are asymptotically sharp for 
small time and in the case of positive curvature also for large time. As a consequence 
it estimated the corresponding spectral zeta function by the Riemann zeta function for the 
surfaces. Specifically, B$\ddot{a}$r \cite{BAR2} derived bounds of the Laplace spectrum 
on closed oriented surfaces. 

In this paper, we extend these results by constructing Kato's bounds for the spectral
 zeta kernel of a perturbed Laplacian of the Schr$\ddot{o}$dinger-type in comparison 
 with spectral zeta kernel of the Laplacian on closed Riemannian manifolds. Our results 
 show that the trace of the operator $\exp (-t H)$  where $H = \Delta _{g} + V$ for smooth 
 potential $V$ is majorised  by the trace of the heat operator heat operator 
 $\exp (-t \Delta _{g})$ on closed and connected  Riemannian manifold $(M,g).$ 
 We used Kato's inequality to prove that the spectral zeta kernel of $H$ is 
 majorised by that of $\Delta _{g}$ on the $n$-dimensional unit sphere.

\section{Preliminaries}
 Consider a Schr$\ddot{o}$dinger-type operator $H = \Delta _{g} +  V$ for some 
 densely defined potential operator $V$ in $L^{2}(M),$  for example, on a smooth tangent bundle 
 $T_{x}M$ of a closed Riemannian manifold $(M,g)$ of dimension $n.$ 
  
 The majorization (domination) of the traces of the semi-group $\e ^{-t H}$ by  $\e ^{-t \Delta _{g}},$  namely
 \begin{equation}\label{kat7}
 \tr \e ^{-t H} \leq  n  \tr \e ^{-t \Delta }, ~~~ t> 0
 \end{equation}
 see for example \cite{Hess} and references therein  leads to the comparison of 
 the spectra of the generators 
 $H$ and $\Delta$ on $M.$  The inequality  (\ref{kat7}) of course yields inequalities 
 for the associated Riemann zeta as well as the spectral zeta functions and the associated 
 spectral zeta kernel on manifolds of certain dimensions. 
 
Limits and integrals; sum and integrals  will be  switched in this paper using the 
dominated convergence and Fubini - Tonelli theorems stated as Theorem (\ref{r0})  
and Theorem (\ref{r1}) below.
\begin{theorem}\label{r0}\cite{RS1}. 
Let $\Omega \subseteq \BR$ be open and let $\{\psi _{k}\}$ be a sequence of 
 integrable functions on  $\Omega .$  Suppose that 
 $\displaystyle{\lim _{k \rightarrow \infty}  \psi _{k} (x) =  \psi (x)}$ $\mu $-almost everywhere.
Further suppose that there exists $\omega \geq 0$ with 
$\displaystyle{\int _{\Omega} \omega (x) \d \mu (x)   < \infty }$ such that 
${\psi _{k}(x) \leq \omega (x) ~~ \forall k} .$ 
Then $\displaystyle{\psi (x) \leq \omega (x) ~~ \mu }$-almost everywhere and 
$\displaystyle{\lim _{k \rightarrow \infty} \int _{\Omega} \psi _{k} (x) 
\d \mu (x) =  \int _{\Omega} \psi (x) \d \mu (x) ;}$ 
where $\d \mu (x)$ is the measure form on $\Omega .$
\end{theorem}

\begin{theorem}\label{r1}\cite{RS1}. 
 Let  $\{\psi_{k}\}$ be a  sequence of  measurable functions. Sum and integral such as  
 $\displaystyle{\sum_{k} \int \psi _{k}(x) \d x }$ can be interchanged in either of 
 the following cases:
\[\psi _{k } \geq 0 , \forall  k \in N ~~ \textrm{or}~~  
 \sum_{k} \int \vert \psi _{k}(x)\vert \d x <\infty .\] 
\end{theorem}  
 
We base our discussion on closed Riemannian manifold $(M,g)$.  To clarify concepts we 
introduce the basic notions for closed Riemannian manifolds through the following definitions. 
\begin{definition} 
A differentiable  $n$-dimensional  manifold $M$ is a connected  paracompact 
Hausdorff topological space for which every point has a neighbourhood 
$\mathcal{U}$ that is  homeomorphic to an open subset $\Omega \subset \BR .$  
 Such a homeomorphism $\psi : \mathcal{U} \rightarrow \OM$ is called a 
chart. Again, a family $\{\mathcal{U} _{\alpha} , \psi _{\alpha}\}$ of charts for which the 
$\mathcal{U}_{\alpha}$ constitute an open covering of $M$ is called an 
atlas. The atlas $\{\mathcal{U} _{\alpha} , \psi _{\alpha}\}$ of $M$ is called 
differentiable if all charts transitions
\[\psi _{\beta} \circ \psi _{\alpha}^{- 1} :  \psi _{\alpha} (\mathcal{U} _{\alpha} \cap   
\mathcal{U} _{\beta} ) \rightarrow  \psi _{\beta} (\mathcal{U} _{\alpha} 
\cap   \mathcal{U} _{\beta})\]
are differentiable of class $C^{\infty} (M).$  A maximal differentiable atlas is called a 
differentiable structure and a manifold with differentiable structure is called 
a differentiable manifold; see  \cite{lee}, \cite{ICH} and \cite{Jost}.
\end{definition}

\begin{definition}
An $n$-dimensional topological manifold with boundary is a Hausdorff  second countable   
topological manifold $M$ in which every point  has a neighbourhood homeomorphic to an open 
subset of the upper half space $H^{n} := \{(x^{1} , \cdots , x^{n}) \in \BR : x^{n}  \geq 0 \}.$
\end{definition}
\begin{definition}
A compact topological manifold whose boundary is empty is called a closed manifold.
\end{definition}
\begin{definition} (Riemannian Manifold): A Riemannian manifold is a pair $(M,g),$ where $g$ 
is a Riemannian metric on the  smooth manifold $M.$
\end{definition}
 
\begin{definition} (Closed Riemannian Manifold): A closed Riemannian manifold  $(M,g)$ is a 
compact Riemannian manifold whose boundary is empty. 
\end{definition}
In the next section we present the basic methods and techniques employed in this work.

\section{Majorisation and the Kato's inequalities}
Majorisation techniques in conjunction with Kato's inequalities are used to compare heat 
operators, the spectral zeta functions and the zeta kernels. Majorisation is a pre-order of 
sequences of real numbers. We make the following formal definitions on majorisation 
and specify its connotation as it is used here.

\begin{definition}\label{majorisation1}
Let $x, y \in \BR$ and let $x^{\downarrow}$ and $y^{\downarrow}$ be vectors with the same 
components as $x$  and $y$ respectively. We say that $x$ weakly majorises $y$ and write 
this as $x \succ _{w} y$ if and only if 
\begin{equation}\label{majeq1}
\sum _{j=1}^{k}  x^{\downarrow} \geq   \sum _{j=1}^{k}  y^{\downarrow} ~~ 
\textrm{for}~~ k =1,2, \cdots , n-1. 
\end{equation}
That is, if $x = (x_{1}, x_{2}, \cdots, x_{n}); ~~y = (y_{1}, y_{2}, \cdots, y_{n}) \in \BR$ 
then $x \succ _{w} y$ if and only if
\begin{eqnarray*}
x_{1} &\geq& y_{1} ,\\
x_{1} + x_{2} &\geq & y_{1} + y_{2} ,\\
&\vdots & \\
x_{1} + x_{2} + \cdots + x_{k} &\geq & y_{1} + y_{2}  + \cdots + y_{k}.
\end{eqnarray*}
Equivalently, if $y$ weakly majorises $x$ we write $x \prec _{w} y$ or $y \succ _{w} x.$

If in addition to (\ref{majeq1}),  we get that 
$\displaystyle{\sum _{j=1}^{n}  x =   \sum _{j=1}^{n}y}$ 
then we say that $x$ majorises $y$ and write this as $x \succeq y .$
\end{definition}

\begin{definition}\label{majorisation2}(\cite{ZS}, \cite{WU}). 
Let $\Omega \subset \BR$ and let  $\psi : \Omega \rightarrow \RR .$ We call  the function
$\psi : \Omega \rightarrow \RR$ Schur convex if 
$x \succeq y$  implies that $\psi(x) \geq \psi (y)$   $\forall x \in \Omega .$
\end{definition}
In this line, we make the following definition.
\begin{definition}\label{majorisation3}  
We say  $x$ majorises/dominates $y$ and write $x \succ y$  whenever (\ref{majeq1}) is satisfied 
and call 
$\psi : \Omega \rightarrow \RR $  Schur convex if $ \psi '' (x) > 0 $   
$\forall x \in \Omega \subset \BR$ and  $\psi (x) \geq \psi (y).$
\end{definition}
We will also use the fact that if $x \succ y$ then $x^{-1} \prec y^{-1};$  see e.g. \cite{TMY},
\cite{ZS} and \cite{WU}.

Kato's inequality plays very salient role in this paper. For any  $\psi \in C^{\infty} (M)$ we 
exploit the Kato's inequality in proving some majorisation results for the spectral zeta kernels. 
We make the following definitions following \cite{Hislop}  and \cite{Are}.
\begin{definition} 
Let $\psi \in C^{\infty} (M)$  be any function. Define the sign function $sgn(\psi)$ by 
\begin{eqnarray} \label{katf1}
sgn  (\psi) (x)  &=&
\left\{ \begin{array}{rcl} 
\bar{\psi}|\psi|^{-1}~ ; & \textrm{if} ~~ \psi(x) \neq 0, \\
 0 ~~~ ~~~ ~;&   \textrm{if} ~~  \psi(x) = 0 
\end{array} \right.
\end{eqnarray}
\end{definition}
For any $\psi , \phi \in C^{\infty} (M)$ the following properties are satisfied:
 \begin{eqnarray}
sgn(\psi) \psi &=& \frac{\bar{\psi} \psi(x)}{|\psi|}  = |\psi|.  \label{katp1}\\
 sgn(\psi) \phi &=& 0 , ~~ \textrm{if}~~ \psi \perp \phi .  \label{katp2} \\
 sgn(\psi) \phi &=& \frac{\bar{\psi} \psi}{|\phi|} ~~ \textrm{if}~~ \psi \not\perp \phi . \label{katp3}\\
| sgn(\psi) \phi | & \leq & |\phi |.  \label{katp4} \\
\lka sgn (\psi ) T \psi , \phi  \rka  & \leq &   \lka |\psi |  , T^{\ast} \phi \rka  \label{katp5}
\end{eqnarray}  
for $T$ a generator of a strongly continuous semigroup such as the heat operator on the 
manifold and where $T^{\ast}$ is the adjoint of $T.$

For any $\epsilon  > 0,$ we also define a regularised  absolute value of $\psi$ by 
\begin{equation}\label{katf2}
\psi_{\epsilon} (x)  := \sqrt{|\psi(x)|^{2} + \epsilon ^{2}} .
\end{equation}
So, $\displaystyle{\lim _{\epsilon \rightarrow 0}  \psi_{\epsilon} (x) = |\psi(x)|}$  point-wise.
 
We also defined the regularised sign function by 
\begin{eqnarray} \label{kat16}
sgn _{\epsilon} \psi &=&
\left\{ \begin{array}{rcl} 
\psi |\psi|^{-1} ~ ; ~& \textrm{on} ~~ supp  (\psi), \\
 \epsilon  ~~~~~~~ ;~ &   \textrm{otherwise .} 
\end{array} \right.
\end{eqnarray} 
for $\psi \in C^{\infty}(T_{x}M). $

We now give the Kato's inequality for the case of the Laplacian.
\begin{proposition}
 Let $\Delta$ be the Laplacian defined by (\ref{kat2}) and let
 $\psi \in C^{\infty} (M).$   Then,  
\begin{equation}\label{katf3}
\Delta _{g} |\psi| \geq \Re \Big( sgn (\psi) \Delta _{g} \psi   \Big)
\end{equation}
except where $|\psi|$ is not differentiable.
\end{proposition}

\begin{proof}
Observe  from (\ref{katf2}) that $\psi _{\epsilon} \geq |\psi|.$  Differentiating 
$\psi_{\epsilon}^{2} = |\psi|^{2} + \epsilon ^{2}$ gives 
\begin{equation}\label{katf4}
\psi_{\epsilon} \nabla _{g} \psi_{\epsilon} = \Re (\bar{\psi} \nabla _{g} \psi ).
\end{equation}
Squaring  (\ref{katf4})  and using that $\psi_{\epsilon} \geq |\psi|$ gives 
\begin{equation}\label{katf5}
|\nabla _{g} \psi_{\epsilon}| \leq  \psi_{\epsilon} ^{-1} |\psi| |\nabla _{g} \psi|  \leq |\nabla _{g} \psi |.
\end{equation}
Now take the divergence of (\ref{katf4})  to obtain 
\[ |\nabla _{g} \psi_{\epsilon}|^{2} + \psi_{\epsilon} \Delta _{g} \psi_{\epsilon} = 
|\nabla _{g} \psi|^{2}  + \Re (\bar{\psi} \Delta \psi ). \]
By  (\ref{katf5}),  this is equivalent to 
\begin{equation}\label{katf6}
 \psi_{\epsilon} \Delta _{g} \psi_{\epsilon} \geq \Re ( \bar{\psi}  \Delta \psi).
\end{equation}
So, using that  $sgn _{\epsilon}\psi = \bar{\psi} |\psi_{\epsilon}|^{-1}$ by (\ref{kat16}) 
so that (\ref{katf6}) becomes 
\begin{equation}\label{katf7}
\Delta _{g} \psi_{\epsilon}  \geq \Re\Big( sgn _{\epsilon} (\psi) \Delta _{g} \psi \Big).
\end{equation}
Now, since $\Delta _{g} \psi_{\epsilon} \rightarrow \Delta _{g} |\psi|$ point-wise and 
$sgn _{\epsilon}(\psi) \rightarrow  sgn (\psi)$ point-wise, take limit in (\ref{katf7}) as 
$\epsilon \rightarrow 0$ to conclude that 
\[ \Delta _{g} |\psi| \geq \Re \Big( sgn (\psi) \Delta _{g}\psi \Big) .\]
except where $|\psi|$ is not differentiable.
\end{proof}

Our next goal is to extend the Kato's inequality  (\ref{katf3}) to a more general  
class of functions. To do this we first recall the following definitions on distributions. 
We begin by recalling that that for any $1 \leq p < \infty ,$ a local $L^{p}$-space 
is defined to be 
$\displaystyle{L^{p}_{loc} (\BR) = \{ f : \int _{\Omega} |f(x)|^{p} \d x < \infty \},}$ 
for any bounded
$\Omega \subset \BR .$  We need that $L^{p}(\BR) \subset L^{p}_{loc} (\BR).$ 
We also recall that if $f \in L^{1}_{loc} (\BR)$ and $g \in C^{\infty}_{0} (\BR )$ then  
$\displaystyle{\int _{\Omega} f(x) \bar{g}(x) \d x < \infty }$ by H$\ddot{o}$lder's inequality; 
see e.g. \cite{AG}, \cite{Hislop} and \cite{kato} for details.

\begin{definition} 
Let $\phi \in  L^{1}_{loc} (\BR)$ and let $\displaystyle{\lka f , g \rka   = 
\int _{\Omega} f(x) \bar{g}(x) \d x < \infty}. $
\begin{itemize}
\item[(1.)]  A function $\psi \in L^{1}_{loc}$ is the distributional derivative of $\phi$ 
with respect to $x_{i} \in \BR$, formally $\psi = \frac{\partial \phi}{\partial x_{i}}$, if 
 $\displaystyle{\lka \psi , f \rka = - \lka \phi , \frac{\partial f}{\partial x_{i}} \rka ~~ 
 \forall  f \in C^{\infty}_{0} (\BR ) .}$ 
 
\item[(2.)]  Let $\phi _{n} , \phi \in L^{1}_{loc}(\BR).$  Then   $\phi _{n}$ converges 
$\phi$ in the distributional sense if  $\displaystyle{\lka \phi _{n} , f \rka  
\rightarrow  \lka \phi , f \rka}$ for all $f \in C^{\infty}_{0} (\BR).$ 

\item[(3.)]  Let $\phi  , \psi \in L^{1}_{loc}(\BR) .$  Then   $\phi  \geq \psi$  in the 
distributional sense if  $\displaystyle{\lka \phi  , f \rka  \geq  \lka \psi , f \rka}$ for all 
$f \in C^{\infty}_{0} (\BR).$ 
\end{itemize}
\end{definition}

\begin{definition}
Let $\omega \in C^{\infty} (\BR)~~ \omega \geq 0,$  and   
$\displaystyle{ \int \omega (x) \d x = 1 .}$ 
For $\epsilon > 0,$  we define 
\[\omega _{\epsilon} (x):= \frac{\omega( \frac{x}{\epsilon})}{\epsilon ^{n}} .\]
Then, 
\[\int \omega _{\epsilon} (x) \d x = 1.\]
We define a map $I_{\epsilon}$ by 
\begin{equation}\label{katf8}
I_{\epsilon}\psi:= \omega _{\epsilon} \ast \psi 
\end{equation}
whenever the right-hand-side of (\ref{katf8}) exists and where 
\[ (f \ast g)(x) := \int f(x-y)g(y) \d y \] 
is the convolution of $f$ and $g$. The map $I_{\epsilon}$ is called an approximation 
of the identity, or simply, an approximate identity. 
\end{definition}

\begin{lemma}\label{al}
Let $I_{\epsilon}$ be an approximate identity. 
\begin{itemize}
\item[(1.)]  If $\psi \in  L^{1}_{loc} (\BR)$ then $I_{\epsilon} \psi \in  C^{\infty} (\BR).$
\item[(2.)]  If $\psi$ is differentiable  then 
$\displaystyle{ \frac{\partial}{\partial x_{i}}(I_{\epsilon} \psi ) 
= I_{\epsilon}  \frac{\partial \psi}{\partial x_{i}}}$; that is, the approximate identity 
 $I_{\epsilon}$ commutes with the differentiation operator 
 $\displaystyle{\frac{\partial}{\partial x_{i}}}.$ 
\item[(3.)] The map $\displaystyle{ I_{\epsilon} : L^{p}(\BR) \rightarrow  L^{p}(\BR)}$ 
is bounded; and 
$|| I_{\epsilon}|| \leq 1.$
\item[(4.)] For any  $\displaystyle{\psi \in  L^{p}(\BR)}$, 
 $\displaystyle{ \lim _{\delta \rightarrow 0} || I_{\epsilon}\psi - \psi|| _{p} = 0.}$ 
\item[(5.)] For any  $\displaystyle{\psi\in  L^{1}(\BR)}$, 
 $\displaystyle{ I_{\epsilon}\psi \rightarrow \psi}$  as $\epsilon \rightarrow 0$ 
 in the sense of distribution. 
\end{itemize} 
\end{lemma}
For proof, one may see \cite{Hislop}.

Now, we have the following theorem for the distributional Laplacian $\Delta _{g}$
 on the a Riemannian manifold $(M,g).$ 
\begin{theorem}
Let $\psi \in  L^{1}_{loc} (M)$ and suppose that the distributional Laplacian
 $\Delta _{g}\psi \in  L^{1}_{loc} (M).$   Then 
\begin{equation}\label{katd1}
\Delta _{g} |\psi| \geq \Re \Big( sgn (\psi) \Delta _{g} \psi \Big)
\end{equation}
in the  sense of distribution.
\end{theorem}

\begin{proof}
Let $\psi \in  L^{1}_{loc} (M).$  By lemma (\ref{al}), $I_{\epsilon}\psi$ 
is smooth for any $\epsilon > 0.$  So, inserting 
 $I_{\epsilon} \psi$ in the theorem in place of $\psi ,$ we obtain for any $\epsilon > 0$ that 
 \begin{equation}\label{katd2}
\Delta _{g} (I_{\epsilon} \psi)_{\epsilon} \geq 
\Re \Big( sgn _{\epsilon} (I_{\epsilon} \psi) \Delta _{g} (I_{\epsilon}\psi)\Big).
\end{equation}
To remove the approximate identity in equation (\ref{katd2}), we use the fact  that since 
$sgn _{\epsilon} (I_{\epsilon} \psi) \Delta _{g} (I_{\epsilon}\psi)$  is a  
sequence  in  $L^{1}_{loc} (M)$, there exists a subsequence of 
$sgn _{\epsilon} (I_{\epsilon}\psi) \Delta _{g} (I_{\epsilon}\psi)$ which converges 
to $sgn _{\epsilon} (\psi) \Delta _{g} (\psi)$ except possibly on a set of measure zero  
(almost everywhere).

Besides, since   $\Delta _{g} \psi \in  L^{1}_{loc} (M),$  it follows from lemma (\ref{al}) 
that  the limit  of  $\Delta _{g} (I_{\epsilon} \psi)$ as $\epsilon \rightarrow 0^{+}$  is 
 $\Delta _{g} \psi \in L^{1}_{loc} (M).$  Also by the boundedness of 
 $ sgn _{\epsilon} (I_{\epsilon} \psi)$ we have that   
 \[\lim _{\epsilon \rightarrow 0^{+}} sgn _{\epsilon} 
 (I_{\epsilon} \psi)\Big(\Delta _{g} (I_{\epsilon}\psi)  
  -  \Delta _{g}\psi  \Big) = 0 \]
 in the distributional  sense.  
 
It is now left to prove that there is a subsequence such that  
$ sgn _{\epsilon} (I_{\epsilon} \psi)\Delta _{g} (I_{\epsilon}\psi)$  converges to 
$sgn _{\epsilon} (\psi\Delta _{g}\psi) .$  
But Lebesgue dominated converges theorem ensures this. Hence, taking this subsequential 
limit in (\ref{katd2}) gives 
\begin{equation}\label{katd3}
\Delta _{g} |\psi |_{\epsilon} \geq 
\Re \Big((sgn _{\epsilon} \psi) \Delta _{g}\psi\Big).
\end{equation}
Therefore,
$\displaystyle{\lim _{\epsilon \rightarrow 0} \Big( \Delta _{g}|\psi |_{\epsilon} \geq 
\Re \Big( (sgn _{\epsilon} \psi) \Delta _{g}\psi \Big)  \Big) }$  gives 
$\displaystyle{ \Delta _{g} |\psi| \geq \Re \Big( sgn (\psi) \Delta _{g} \psi \Big) }$ 
in the  sense of distribution as required.
\end{proof}

\begin{corollary}
Suppose $\psi , ~ \Delta \psi \in L^{2} (M)$ then 
\begin{equation}\label{katlap3}
\Delta |\psi| \geq \Re \Big( \bar{\psi} \psi ^{-1} \Delta \psi  \Big).
\end{equation}
\end{corollary}
\begin{proof}
Since $\psi , ~ \Delta \psi \in L^{2} (M)$ then there exists a sequence 
$\{\psi _{j}\} \in \mathcal{D}$ which converges in $\mathcal{D}$ such that in $L^{2}$-norm 
\[\nabla \psi _{j} \rightarrow  \nabla \psi ~~ \textrm{and}~~ \Delta \psi _{j}
 \rightarrow  \Delta \psi \]
Moreover, 
\[ ||\nabla \psi ||^{2} = \lka \nabla \psi , \nabla  \psi \rka _{g}
= - \lka \psi , \Delta \psi \rka  _{g}  \leq ||\psi ||_{2} ||\Delta \psi ||_{2} .\]
By the continuity of these estimates, it follows that 
$\displaystyle{\Delta |\psi| \geq \Re \Big( sgn (\psi) \Delta \psi  \Big)}$ for all 
$\psi \in L^{2} (M).$
\end{proof}

We can extend this concept to  semigroups. Let $T$ be a strongly continuous semigroup with
 differential operator generator $A$ defined by 
\[Ax = \lim _{t\downarrow 0} \frac{1}{t}\Big(T(t) - id  \Big)x \]
for $x \in M.$ That is $T$ satisfies 
\begin{enumerate}
\item $T(0) = id$ that is an identity operator on $M.$
\item For all $0 \leq t,s \in M,$  $T(t+s) = T(t)T(s).$
\item For all $x_{0} \in M ,$ $||T(t) x_{0} - x_{0}|| \rightarrow 0$ as $t \downarrow 0.$ 
That is the strong operator topology. See e.g. \cite{AG}  and \cite{RS1}.
\end{enumerate}

We have the following result.
\begin{theorem}\label{kkt1}
The semigroup $T$ satisfies Kato's inequality, that is,
$\displaystyle{\lka sgn(\psi)A \psi , \phi \rka \leq  \lka |\psi| , A^{\ast}\phi \rka}$ 
for $\psi  \in \textrm{domain}(A),$  $\phi \in \textrm{domain}(A^{\ast})$ where $T^{\ast}$
 is the adjoint of $A.$
\end{theorem}

\begin{proof} 
Let  $\psi \in \textrm{domain}(A)$ and $\phi \in \textrm{domain}(A^{\ast})$ then 
\begin{eqnarray*}
\lka sgn(\psi)A \psi , \phi \rka &=& 
\lim _{t \downarrow 0}\frac{1}{t} \lka sgn(\psi)(T(t)\psi - \psi) , \phi \rka  \\
&=& \lim _{t \downarrow 0}\frac{1}{t} \lka sgn(\psi)T(t)\psi -| \psi | , \phi \rka  \\
&\leq & \lim _{t \downarrow 0}\frac{1}{t} \lka | T(t)\psi | -| \psi | , \phi \rka  \\
&\leq & \lim _{t \downarrow 0}\frac{1}{t} \lka |\psi | , (T^{\ast}(t)\phi  - \phi) \rka  
=  \lka |\psi| , T^{\ast}\phi \rka .
\end{eqnarray*}
\end{proof}
Of course, in the case that $A$ is self-adjoint, the the Kato's inequality for the semigroup is 
\[\lka sgn(\psi)A \psi , \phi \rka \leq  \lka |\psi| , A\phi \rka .\]

Now consider a Schr$\ddot{o}$dinger type operator $H = \Delta + V$  on $C^{\infty}(M)$ 
where $V$ is a nonnegative multiplication potential. We have the following immediate result.
\begin{theorem}
For all $\psi \in C^{\infty}(M),$ 
$\displaystyle{H |\psi| \geq \Re \Big( sgn (\psi ) V \psi \Big).}$ 
\end{theorem}

\begin{proof}
By definition of  $I_{\epsilon}$ we have 
\begin{equation}\label{katlap5}
H  I_{\epsilon} |\psi| = \Delta  I_{\epsilon} |\psi| + V I_{\epsilon} |\psi| 
 = I_{\epsilon} (\Delta   |\psi| + V |\psi|)  \geq 0 
\end{equation}
since $\Delta |\psi|  + V |\psi| \geq 0 .$  But,  
\begin{equation}\label{katlap6}
 \lka  \Delta  I_{\epsilon} |\psi| ,   I_{\epsilon} |\psi| \rka _{g}   = 
- || \nabla (I_{\epsilon} |\psi|) ||^{2}_{2}  \leq 0 .
\end{equation}
Since by (\ref{katlap5})  the left side of (\ref{katlap6}) is nonnegative then 
$\nabla (I_{\epsilon} |\psi|) =0$  in the $L^{2}$-sense. \\ 
Therefore $I_{\epsilon} |\psi| = V|\psi| = c \geq 0$ with $c$ a constant.

Since $|\psi \in L^{2}(M)|$ and  $|I_{\epsilon} |\psi | \rightarrow  |\psi|$ in the 
$L^{2}$-sense, we conclude that $c=0$ and so $I_{\epsilon} |\psi | =0,$  $|\psi| =0$ 
and so $\psi =0.$ 

Hence, 
$\displaystyle{ H  |\psi | \geq \Delta |\psi | \geq \Re \Big( sgn(\psi) \Delta \psi \Big) \geq  
\Re \Big( sgn(\psi) V \psi \Big)   .}$
\end{proof}

Let $X$ and $Y$ be operators on Hilbert space of functions $\mathcal{H}.$ 
We assume that $X$ and $Y$ are well defined generators of the heat semigroups $\e ^{-tX}$  
and $\e ^{-tY}$ satisfying 
\begin{eqnarray}\label{hkm1}
\left.\begin{array}{rcl}
\big(\frac{\partial}{\partial t}    +  X \big) \e ^{-tX} f&= & 0  \\
\displaystyle{\lim _{t \rightarrow 0}  \e ^{-tX} f } & = & f
\end{array}\right\}
\end{eqnarray}
and similarly 
\begin{eqnarray}\label{hkm2}
\left.\begin{array}{rcl}
\big(\frac{\partial}{\partial t}    +  Y \big) \e ^{-tY} f&= & 0  \\
\displaystyle{\lim _{t \rightarrow 0}  \e ^{-tY} f } & = & f
\end{array}\right\}
\end{eqnarray}
for some $f \in \mathcal{H}.$ In the theorem that follows, we denote 
$\e ^{-tX} f$ just  by $\e ^{-tX}.$

\begin{lemma}\label{km2a}
If $X$ is the Laplacian  $\Delta$ in (\ref{kat2}) then $\Delta \e ^{-t \Delta}  
= \e ^{-t \Delta} \Delta .$ 
\end{lemma}
\begin{proof}
A direct computation shows this. That is, 
\begin{eqnarray*}
\Delta \e ^{-t \Delta} f(x) &=& \Delta _{x} \Big( \int _{M}K(t,x,y) f(y) dV(y) \Big)\\
&=&  \int _{M}\Delta _{x} K(t,x,y) f(y) dV(y)  \\
&=& - \partial _{t}\int _{M} K(t,x,y) f(y) dV(y) 
\end{eqnarray*}
and by symmetry of $K(t,x,y)$ in $x$ and $y$ we have 
\begin{eqnarray*}
\e ^{-t \Delta} \Delta f(x) &=&  \int _{M}   K(t,x,y) \Delta _{y} f(y) \d V(y) \\
&=&  \int _{M}\Delta _{y} K(t,x,y) f(y) \d V(y)  \\
&=& - \partial _{t}\int _{M} K(t,x,y) f(y) \d V(y)
\end{eqnarray*}
which proves the lemma.
 \end{proof}

\begin{theorem}\label{hkm3}
The heat semigroup $\e ^{-t(X + Y)}$ satisfies the Duhamel's formula
\begin{equation}\label{hkm4}
\e ^{-t(X + Y)} = \e ^{-tX} - \int _{0}^{t}  \e ^{-(t-s)(X + Y)}Y \e ^{-sX} \d s .
\end{equation}
\end{theorem}
For proof, one can see e.g. \cite{Rose} or \cite{CFK}.

\section{The spectral zeta function and kernel} 
The Riemann zeta function $\zeta_{R}$ is the function defined as 
$\displaystyle{\zeta_{R}: \{ s \in \CC : \Re(s) >1 \} \rightarrow \CC }$ with 
\begin{equation}\label{hz1}
\zeta _{R}(s) = \sum _{k = 1}^{\infty} \frac{1}{k ^{s}} ;~~ \Re (s) > 1.
\end{equation}
From the Riemann zeta function (\ref{hz1}),  notice that since 
\begin{equation}\label{n3}
\sum _{k=1}^{\infty} \big\vert \frac{1}{k^{s}} \big\vert 
= \sum _{k=1}^{\infty} \frac{1}{k^{\Re(s)}} ,
\end{equation}
the series on the right-hand-side of  (\ref{n3}) converges absolutely if and only 
if $\Re (s) > 1.$  The Riemann zeta function defined by (\ref{hz1}) above is holomorphic in the 
region indicated. It,  however,  admits a meromorphic continuation to 
the whole $s$-complex plane with only simple pole at $s=1$ and has residue $1.$  
For details, see \cite{OM16}.  

Hurwitz zeta function $\zeta _{H} (s, a)$ is a generalization of the Riemann 
zeta function  (\ref{hz1}). It is defined below.
\begin{definition}
Let $s \in C $ and $0< a \leq 1 .$ Then for $\Re (s) > 1 ,$ the Hurwitz zeta 
function  is defined by 
\begin{equation}\label{hz2}
\zeta _{H}(s, a) = \sum _{k = 0}^{\infty} \frac{1}{(k + a)^{s}} ;~~ \Re (s) > 1.
\end{equation}
\end{definition}
Clearly, $\zeta _{H}(s, 1) =  \zeta _{R}(s) .$

Another generalisation of the Riemann zeta function is the spectral 
zeta function, which is  the function of interest in this paper. 
The spectral zeta function is explicitly defined through  the operator $\Delta _{g} ^{- s}$ 
and its integral kernel $\zeta_g(s,x,y)$, also called the zeta kernel. 
 The operator $ \DEL _{g} ^{- s}$ is
uniquely defined by the following properties (see e.g \cite{OM16}):
\begin{itemize}
\item[(1.)] it is linear on $L^{2}(M)$ with 1-dimensional null space consisting of 
constant functions. 
This ensures that the smallest eigenvalue of $ \DEL _{g} ^{- s}$ is $0$ of 
multiplicity $1$  with corresponding eigenfunction $\frac{1}{\sqrt{V}}$ where 
$V$  is the volume of $M;$
\item[(2.)] the image of $\Delta _{g} ^{- s}$ is contained in the orthogonal 
complement of  constant functions in $L^{2}(M)$ i.e.
\[\int _{M} \DEL _{g} ^{- s} \psi \d V_{g} = 0 ~~ \forall ~ \psi \in L^{2}(M)~~
 \mathrm{constant};~~ \mathrm{and} \]
\item[(3.)] $\DEL _{g} ^{- s} \psi _{k} (x) = \la _{k} ^{- s}\psi _{k}(x)$ for all 
$\psi _{k};~ k > 0$  an orthonormal basis of eigenfunction of $\Delta _{g} .$
\end{itemize}
Then for  $\Re(s) > \frac{n}{2}$, we see by property (3.) that $ \Delta _{g} ^{- s}$ is 
trace class, with trace given by
the spectral zeta function, namely
\begin{equation}\label{3}
\zeta _{g} (s) =\sum _{k = 1} ^{\infty}
 \frac{1}{\la _{k} ^{s}}= \tr (\Delta _{g} ^{- s}) =  \int _{M} \zeta_g (s,x,x) \d V ~;~
   ~~~\Re (s) >\frac{n}{2}.
\end{equation}

\begin{theorem} \cite{SM}.  Let $\{\psi _{k}\}_{k=1}^\infty$ be an orthonormal eigenbasis 
for $\Delta_g$ corresponding to the eigenvalues $\{ \la_k \}_{k=1}^\infty$ listed with 
multiplicities. Then the zeta kernel,  $\zeta _{g} (s,x,y)$, also called the point-wise zeta 
function, is equal to
\begin{equation}\label{4}
\zeta _{g} (s,x,y) = \sum _{k =1} ^{\infty} \frac{\psi _{k} (x)
 \bar{\psi}_{k} (y)}{\la _{k} ^{s}} ; ~~ \Re(s) > \frac{n}{2}.
\end{equation}
\end{theorem}
For proof, see e.g. \cite{OM16} and \cite{MMT}.

From here on, we suppress the subscript $g$ in $\zeta _{g} (s)$  and  $\Delta _{g}.$  
 We simply write $\zeta  (s)$  and $\Delta $ for   $\zeta _{g} (s)$  and  $\Delta _{g}$ 
 respectively,  unless for purpose of emphasis. 

A relationship between the zeta kernel and the heat kernel enables to define the 
spectral zeta kernel explicitly. The heat kernel, 
$ K(t,x,y) : (0, \infty) \times M \times M \rightarrow R , $
is a continuous function on $(0, \infty) \times M \times M .$  It is the so-called 
fundamental solution to the heat equation, i.e, it is the unique solution to the 
following system of equations:
\begin{equation}\label{e1}
\left.\begin{array}{rcl}
\big(\frac{\partial}{\partial t}    +  \DEL _{x}\big)K(t,x,y) &= & 0  \\
\displaystyle{\lim _{t \rightarrow 0} \int _{M}  K(t,x,y) \psi (y) \d V_{y}} & = & \psi(x)
\end{array}\right\}
\end{equation}
for  $t >0;~x,y \in M$ and  $\Delta _{x}$ is the Laplacian acting on any 
 $\psi ~\in ~ L^{2}(M) ,$ where the limit in the second equation of (\ref{e1})  
 is uniform for every $\psi \in C^{\infty}(M).$  

The heat operator $ \e ^{- t \Delta} : L^{2} (M) \rightarrow L^{2} (M)$
is the operator defined by the integral kernel $K(t,x,y)$ as
\[(\e ^{- t \Delta} \psi )(y) :=  \int _{M} K(t,x,y) \psi (x) \d V_{x}\] 
for $\psi \in L^{2} (M)$.
The heat kernel is symmetric in the space variables, that is 
$K(t,x,y) =  K(t, y , x) ~~ \forall ~ x,y \in M .$  Thus the heat operator is self-adjoint, 
that is, for $\psi _{1} ,\psi _{2} \in  L^{2} (M)$  we have
\begin{eqnarray*}
 \lka  \e ^{- t \Delta} \psi _{1} , \psi _{2}  \rka _{L^{2}(M)}  &=&
\int _{M} \big\{ \int _{M} K(t,x,y) \psi _{1}(y) \d V_{y} \big \}  
\bar{\psi }_{2}(x) \d V_{x} \\
&=&   \int _{M} \big\{ \int _{M} K(t,y,x) \psi _{2}(x) \d V_{x} \big \} 
 \bar{\psi}_{1}(y) \d V_{y} =  \lka  \psi _{1} ,
   \e ^{- t \Delta} \psi _{2}  \rka  _{L^{2}(M)}. 
 \end{eqnarray*}

Now returning to the heat kernel,  let $\displaystyle{\{ \psi_{k} \}_{k=0}^{\infty}}$ with 
$\displaystyle{\int _{M} \psi_{k}(x) \bar{\psi}_{l}(x) \d V_{g}(x) = \delta _{kl}}$ 
be orthonormal basis of eigenfunctions of $\Delta$ with corresponding eigenvalues
 $\{\lambda _{k} \}$  listed with multiplicities. Then  $\{ \psi _{k} \}_{k=0}^{\infty}$  
 are also eigenfunctions  of the heat operator with  corresponding eigenvalues 
 $\{ \e ^{- \la _{k} t} \}.$  In terms of these eigenfunctions,  the Mercer's theorem 
 implies that $\e ^{- t \Delta}$ is trace-class for all $t > 0$ and  one can write the 
 heat kernel  as 
\[ K(t,x,y) = \sum _{k=0}^{\infty} \e ^{- \la _{k} t} \psi_{k} (x) \bar{\psi}_{k}(y) . \]
The convergence for all $t > 0$ is uniform on $M \times M .$  
In particular, the trace of the heat operator
\begin{equation}\label{5}
\tr(\e ^{- \Delta _{g} t}) =  \sum _{k=0}^{\infty} \e ^{- \la _{k} t} |\psi_{k} (x) |^{2}  
= \sum _{k=0}^{\infty} \e ^{- \la _{k} t} = \int _{M} K(t,x,x) \d V_{g}(x)  < \infty .
\end{equation}
\begin{lemma}\label{l2}
The zeta kernel and the heat kernel are related by
\[\zeta_{g}(s,x,y) = \frac{1}{\Gamma (s)} \int _{0}^{\infty} t^{s-1} (K(t,x,y)
 - \frac{1}{V})\d t,\]
$\Re(s) > \frac{n}{2}.$
\end{lemma}
\begin{proof}
Observe that for any $x > 0$ and $\Re (s) > 0 ,$
\[x ^{- s} = \frac{1}{\Gamma (s)} \int _{0}^{\infty} e^{- x t} t^{s-1} \d t \] 
since a change of variable from, say, $x t$ to $\tau $ gives $ x ^{- s}$ and since 
$\Gamma (s)$ is holomorphic for $\Re (s) > 0 .$

Consequently,
\[\la _{k} ^{- s} = \frac{1}{\Gamma (s)} \int _{0}^{\infty} e^{- \la _{k} t}
  t^{s-1} \d t  .\]
Thus,
\[\zeta_{g}(s,x,y) = \sum _{k=1}^{\infty}\left[ \psi_k(x) \overline{\psi}_k(y)
 \frac{1}{\Gamma (s)} \int _{0}^{\infty} t^{s-1} e^{- \la _{k} t} \d t
 \right];~~~~\Re(s) > \frac{n}{2} .\]
Therefore, using Theorem \ref{r1} to switch the order of the sum and 
the integral, we have
\begin{eqnarray*}
\zeta_{g}(s,x,y) &=&   \frac{1}{\Gamma (s)} \int _{0}^{\infty} 
\left(\sum _{k=1}^{\infty} e^{- \la _{k} t} \psi_k(x) 
\overline{\psi}_k(y)   \right) t^{s-1} dt.
\end{eqnarray*}
Thus,
 \begin{equation}\label{zetaint}
 \zeta_{g}(s,x,y) = \frac{1}{\Gamma (s)} \int _{0}^{\infty} t^{s-1}
  (K(t,x,y) - \frac{1}{V} )dt .
 \end{equation}
\end{proof}

We have another result as the corollary that follows. 
\begin{corollary}\label{corsa1}
The Schr$\ddot{o}$dinger-like operator $H = \Delta + V$  where $V \in L^{2}_{loc}$  and 
$V \geq 0$  is essentially self-adjoint on $C^{\infty}_{0}(M)$.
\end{corollary}

\begin{proof}
Since  the domain $D(H^{\ast}) \subset L^{2}(M),$ it suffices to show that 
$\ker (H^{\ast} + 1) = \{ 0 \}.$ This implies that if 
\begin{equation}\label{corsa2}
(\Delta + V + 1) \psi  = 0,~ ~ \textrm{for} ~~ \psi \in L^{2}(M)
\end{equation}
then $\psi = 0.$ We prove (\ref{corsa2})  by Kato's inequality.  Since $\psi \in L^{2}(M)$  and 
$V \in L^{2}_{loc}(M)$ it follows by Cauchy-Schwarz inequality that $V \psi \in L^{1}_{loc} (M)$ 
following the inclusion $L^{2} \subset L^{2}_{loc} \subset L^{1}_{loc}$ from the estimate 
\[ \int _{M} 1 \cdot |\psi (x)| \d V_{g}  \leq  V_{g} \sqrt{\int _{M} 
 \cdot |\psi (x)|^{2} }\d V_{g}\]
where $V_{g}$ is the volume of $M$. This implies that $\psi \in L^{1}_{loc} (M).$ 

Using the Kato's inequality, we have 
\begin{eqnarray*}
\Delta |\psi| \geq \Re \Big(( sgn \psi)\Delta \psi  \Big)
\geq \Re \Big(( sgn \psi)(V+1) \psi  \Big)
= |\psi |(V+1)  \geq 0.
\end{eqnarray*}
Hence, the function $\Delta |\psi| \geq 0$ and so, 
\begin{equation}\label{corsa3}
\Delta I_{\epsilon}|\psi| = I_{\epsilon}\Delta |\psi| \geq 0.
\end{equation}
On the other hand, $I_{\epsilon}|\psi| \in D(\Delta)$ and therefore
\begin{equation}\label{corsa4}
\lka \Delta (I_{\epsilon}|\psi|), (I_{\epsilon}|\psi|) \rka  =
-||\nabla( I_{\epsilon} |\psi|) ||^{2} \leq 0.
\end{equation}
But by equation (\ref{corsa3}) the left side of (\ref{corsa4}) is nonnegative and so   
$\nabla( I_{\epsilon} |\psi|) =0$ in the $L^{2}$-sense  and therefore 
$I_{\epsilon} |\psi| = c \geq 0.$
But $|\psi| \in L^{2}$ and $ I_{\epsilon} |\psi| \rightarrow |\psi|$ 
in the $L^{2}$-sense and so $c =0.$  Hence $I_{\epsilon} |\psi| =0$ and so $|\psi| =0$  
and $\psi =0.$
\end{proof}

\section{Bounds for spectral kernels on the $n$-sphere}
Consider the Laplacian on the unit $n$-dimensional sphere 
$S^{n} = \{ x \in \mathbb{R}^{n + 1} : \| x\|  = 1 \}$ 
defined  in polar coordinates as  
 \begin{equation}\label{splap}
\Delta _{n} = \frac{1}{\sin ^{n -1} \theta}  \frac{\partial }{\partial \theta}
 \{ \sin ^{n - 1} \theta \frac{\partial }{\partial \theta} \} 
 + \frac{1}{\sin ^{2} \theta} \Delta _{n -1} 
 \end{equation}
 where  $\Delta _{n -1}$ is  the  Laplacian  on  $S^{n - 1} .$

The harmonic homogeneous polynomials restricted to the $n$-sphere are the eigenfunctions of the
 Laplacian on $S^{n} .$  A  detailed treatment of these functions can be found in 
 \cite{MMT} and \cite{KAW}. The restriction  of  elements  of 
$\mathcal{H} _{k}$  to $ S^{n} $ are called spherical harmonic polynomials of degree $k$, 
and are therefore eigenfunctions of $ \Delta _{n} $ with eigenvalues $k ( k + n-1) .$ 
 
The dimension $d_{k}(n)$ of the space of harmonic polynomial  $\mathcal{H}_{k}$ is 
given by the formula 
\begin{equation}\label{ef6}
d_{k} (n) = 
\begin{pmatrix}
k + n \\
n
\end{pmatrix}
-
\begin{pmatrix}
k + n - 2 \\
n
\end{pmatrix} = \frac{(2k + n - 1)(k + n - 2)!}{k! (n - 1)!}
\end{equation}
where $k \in N_{0} ~ \textrm{and}  ~ n \geq 1 
~ \mathrm{is ~  the ~ dimension ~ of ~ the ~ manifold} ~~ S^{n}.$
For proof, one may see \cite{OM16}.

The zeta and heat kernels  can be expressed  explicitly for the  $n$-sphere in terms of 
Gegenbauer polynomials. The Gegenbauer polynomials are the generalisations of the Legendre 
polynomials to higher dimensions. These polynomials can be characterised by a formula 
generalising the Rodrigues representation of  the  Legendre polynomials, namely 
\begin{equation}\label{gph0}
P_{m}(x)= \frac{1}{2^{m}m!} \frac{d^{m}}{dx^{m}}(x^{2} - 1)^{m} .
\end{equation}
The polynomials (\ref{gph0}) solve the Legendre equation 
\begin{equation}\label{gph1}
(1-x^{2})P_{m}(x)^{\prime \prime} - 2x P_{m}(x)^{\prime} + (k(k + 1) )P_{m}(x) = 0.
\end{equation}

Gegenbauer polynomials are relevant to the study of the heat and zeta kernels because
of addition formula stated below. 
\begin{lemma}(Addition formula, c.f: Morimoto \cite{MMT}) \label{addf}\\
Let $\{ \psi _{k,j}: 1 \leq j \leq  d_{k}(n)  \}$ be an orthonormal basis of the
 space of  $n$-dimensional spherical harmonics  $\mathcal{H}_{k}(S^{n}),$ i.e:
\begin{equation}\label{gph10a}
\int _{S^{n}}\psi _{k,j}(x) \bar{\psi }_{k,l}(x) \d V_{g}(x) = \delta _{jl};~~1 
\leq j,l \leq d_{k}(n).
\end{equation}
Then
\begin{equation}\label{gph11}
\sum _{j = 1}^{d_{k}(n)} \psi _{k,j}(x) \bar{\psi }_{k,l}(y) = 
\frac{d_{k}(n)}{|S^{n}|} P_{k}^{\frac{(n - 1)}{2}}(x \cdot y)
\end{equation}
where as before, $P_{k}^{n}(t)$ are the Gegenbauer polynomials of degree  
$k$ in $n$ dimensions.
\end{lemma}
For proof, one may see Morimoto \cite{MMT}. Note in particular, this means that 
$P_{k}^{(n-1)/2}(x\cdot y)$ is a harmonic function on $S^n$  with eigenvalue 
$\lambda_k= k(k+n-1)$.  

The Gegenbauer polynomials enable one to write the heat kernel on $S^{n}$ 
explicitly namely, for all $t > 0$, and $x, y \in S^{n}$: 
\begin{eqnarray}\label{gph9}
K(t,x,y) &:=&\frac{1}{V_{n}}\sum _{k = 0}^{\infty} \sum _{j = 1}^{d_{k}(n)} 
\e ^{-k(k + n - 1)t} \psi _{k,j}(x) \bar{\psi }_{k,j}(y)\\
&=& \frac{1}{V}\sum _{k = 0}^{\infty} \e ^{-k(k + n - 1)t}\frac{d_{k}(n)}
{P_{k}^{\frac{(n -  1)}{2}}(1)}P_{k}^{\frac{(n - 1)}{2}}(x \cdot y).
\end{eqnarray}
where $V_{n}$ is the volume of $S^{n},$ and 
$d_{k}(n)$ is the dimension of the $\lambda_k$ eigenspace. 
It is also known that the zeta kernel $\zeta _{S^{n}} (s, x, y)$ on $S^{n}$ is 
explicitly given by 
\begin{equation}\label{zetak}
\zeta _{S^{n}} (s, x, y) = \frac{1}{V_{n}} \sum _{k = 1}^{\infty} 
\frac{d_{k}(n)}{(k(k+ n - 1))^{s}} \cdot  \frac{1}{P_{k}^{\frac{(n - 1)}{2}}(1)}
P_{k}^{\frac{(n - 1)}{2}}(x \cdot y)
\end{equation}
(see e.g \cite{MMT}).

We now consider the Schr$\ddot{o}$dinger operator $\Delta _{g} + c $ where  $c = \frac{n-1}{2}$ 
and $n$ is the dimension of the unit sphere $S^{n}.$  One expresses the associated spectral zeta 
function in terms of the Hurwitz zeta function. Denote by $\{\mu _{k} \}$ the spectrum of 
$\Delta _{g} + c $  on  $S^{n}:$
\begin{equation}
\mu _{k} =  k(k+ n -1) + c
\end{equation}
with the same eigenfunctions and multiplicities, $d _{k}(n)$  as for $\DEL_{g}$. 

We define the regularized zeta function as 
\begin{equation}\label{spec3c}
Z_{S^{n}}(s) = \sum_{k=1}^{\infty}  \frac{d_{k}(n) }{\mu _{k}^{s}}  =
\sum_{k=1}^{\infty}  \frac{d_{k}(n) }{(k + \frac{n - 1}{2})^{ 2s}}    
;~~~~\Re (s)>\frac{n}{2} .
\end{equation}

The regularized zeta function  (\ref{spec3c}) of the  operator $\DEL _{g} + c$  
on $S^{n}$ can  then  be  expressed in terms of  the Riemann zeta function following 
Elizalde and others e.g. \cite{EE}. 
For $n = 1,$  the regularized zeta function on the unit circle becomes 
\[Z_{S^{1}} (s) =  \sum_{k=1}^{\infty}  \frac{2}{k ^{ 2s}} = 2 \zeta _{R} (2 s)\] 
since $d_{k}(1) = 2 .$
On $S^{2} $  and $S^{3} $ the spectra are shifted by $\frac{1}{4}$   and $1$ with 
$d_{k}(2) = 2k + 1 $ and $d_{k}(3) = (k + 1)^{2} $ respectively. Continuing this way gives
\begin{eqnarray*}
  Z_{S^{2}} (s) &=&  \sum_{k=1}^{\infty}  \frac{2k + 1}{(k + \frac{1}{2} )^{ 2s}}
  = (2^{2 s} - 2 )\zeta _{R} (2 s - 1)   - 4^{s} ;\\
   Z_{S^{3}} (s) &=& \sum_{k=1}^{\infty}  \frac{(k + 1)^{2}}{(k + 1)^{ 2s}}
 = \zeta _{R} (2 s - 1)   - 1 ;\\
 Z_{S^{4}} (s) &=& \frac{1}{6} \sum_{k=1}^{\infty}  \frac{(k + 1)(k + 2)(2k + 3)}{(k 
+ \frac{3}{2})^{ 2s}}\\
&=& \frac{1}{3}(2^{2 s - 3} - 1 )\zeta _{R} (2 s - 3)  
 - \frac{1}{3}(2^{2 s - 3} - \frac{1}{4} )\zeta _{R} (2 s - 1) - \frac{1}{3} 
 (\frac{2}{3})^{2s - 3} + \frac{1}{8} (\frac{2}{3})^{2s}.
\end{eqnarray*}

Now we can write $Z _{S^{n}}(s)$ as 
\begin{equation}\label{mainn1}
Z _{S^{n}}(s, c) = \sum _{k = 0}^{\infty} \frac{1}{(k + c)^{2s}} ; 
\textrm{with}~~c = \frac{n-1}{2}~~
\textrm{for} ~~ \Re (s) > \frac{n}{2}.
\end{equation}
Clearly, $Z _{S^{n}}(s, 1) =  \zeta _{S^{n}}(2s) .$

Finally we prove that $Z _{S^{n}}(s, c) \prec \zeta _{S^{n}}(s) .$ 
This is the theorem that follows.
\begin{theorem}\label{mainn2}
Let $Z _{S^{n}}(s, \rho _{n})$ be the spectral zeta kernel of the 
Schr$\ddot{o}$dinger-type operator, $H = \Delta + c,$ on smooth functions of the 
$n$-dimensional unit sphere, $S^{n},$ where $c$ is the potential operator that multiplies 
by $\displaystyle{\rho _{n}=\frac{n-1}{2}}.$ Then 
\begin{equation}\label{mainn3}
Z _{S^{n}}(s, x,y) \prec \zeta _{S^{n}}(s,x,y) .
\end{equation}
\end{theorem}

\begin{proof}
Let $\{\psi _{k,j}: 1 \leq j \leq  d_{k}(n)\}$ be an orthonormal basis of the
 space of  $n$-dimensional spherical harmonics  $\mathcal{H}_{k}(S^{n}).$ 
By Kato's inequality of Theorem \ref{kkt1} it suffices we use  
\begin{equation}\label{mainn4}
\lka sgn(\psi _{k,j})H \psi _{k,j}, \psi _{k,l}\rka \leq \lka |
\psi _{k,j}|,\Delta \psi _{k,l}\rka .
\end{equation}
So,
\begin{eqnarray*}
\lka sgn(\psi _{k,j})H \psi _{k,j}, \psi _{k,l}\rka &=&  
\lka sgn(\psi _{k,j})(\Delta _{x} + c) \psi _{k,j}, \psi _{k,l}\rka  \\
&=& \lka sgn(\psi _{k,j})\Delta _{x}\psi _{k,j}, \psi _{k,l}\rka  + 
\lka sgn(\psi _{k,j}) c \psi _{k,j}, \psi _{k,l}\rka \\
&=&
\lambda _{k}\lka sgn(\psi _{k,j})\psi _{k,j}, \psi _{k,l}\rka  + 
\rho _{n} \lka sgn(\psi _{k,j})\psi _{k,j}, \psi _{k,l}\rka \\
&=&  (\lambda _{k}+\rho _{n}) \lka sgn(\psi _{k,j})\psi _{k,j}, \psi _{k,l}\rka .
\end{eqnarray*}

Hence, 
\[ \int_{S^{n}} \Delta ^{-s} \lka sgn(\psi _{k,j})H \psi _{k,j}, \psi _{k,l}\rka \d V_{y} 
=  \int_{S^{n}} \Delta ^{-s}(\lambda _{k}+\rho _{n}) \lka sgn(\psi _{k,j}) \psi _{k,j},
 \psi _{k,l}\rka \d V_{y} .\]
Since $\Delta ^{-s}$ is trace class with trace given by \ref{3} we have 
\begin{eqnarray*}
 \int_{S^{n}} \Delta ^{-s} \lka sgn(\psi _{k,j})H \psi _{k,j}, \psi _{k,l}\rka \d  V_{y} 
&=&  \sum _{k=1}^{\infty} \frac{1}{(\lambda _{k}+\rho _{n})^{s}} 
\int_{S^{n}}  \lka sgn(\psi _{k,j}) \psi _{k,j}, \psi _{k,l}\rka \d  V_{y}  \\
&\leq & \sum _{k=1}^{\infty} \frac{1}{(\lambda _{k}+\rho _{n})^{s}} 
\int_{S^{n}}  \psi _{k,j} \bar{\psi} _{k,l} \d V_{y}  \\
&=& \sum _{k=1}^{\infty} \frac{1}{(\lambda _{k}+\rho _{n})^{s}} .
\end{eqnarray*}
However, for $|z| < 1$ the following binomial expansion holds
\[(1 - z)^{-2s} = \sum _{m = 0}^{\infty} \frac{\Gamma(2s + m)}{m!\Gamma(2s)} z^{m} .\]
So for $\Re (2s) > 1 ,$ we have 
\begin{eqnarray*}
\zeta _{H}(2s, \rho _{n}) &=& \frac{1}{\rho _{n}^{2s}} + \sum _{k = 1}^{\infty} 
\frac{1}{k ^{2s}} \frac{1}{(1 + \frac{\rho _{n}} {k})^{2s}} \\
& = & \frac{1}{\rho _{n}^{2s}} + \sum _{k = 1}^{\infty} \frac{1}{k ^{2s}}
\sum _{m = 0}^{\infty} (-1)^{m} \frac{\Gamma(2s + m)}{m!\Gamma(2s)}
(\frac{\rho _{n}}{k})^{m} \\
&= & \frac{1}{\rho _{n}^{2s}} + 
\sum _{m = 0}^{\infty} (-1)^{m} \frac{\Gamma(2s + m)}{m!\Gamma(2s)} (\rho _{n})^{m}
\sum _{k = 1}^{\infty} \frac{1}{k ^{2s + m}}
\end{eqnarray*}
which gives the expansion
\[Z_{S^{n}}(s, \rho _{n}) = \frac{1}{\rho _{n}^{2s}} + \sum _{m = 0}^{\infty} (-1)^{m}
 \frac{\Gamma(2s + m)}{m!\Gamma(2s)} \rho _{n}^{m} \zeta _{R}(2s + m)\]
 provided $0 < \rho _{n} \leq 1$ and where $\zeta _{R}$ is the Riemann zeta function.
 
 Thus, since 
 \[\displaystyle{\frac{d_{k}(n)\downarrow}{\mu ^{s}_{k}}} \preceq 
 \displaystyle{\frac{d_{k}(n)\downarrow}{\lambda ^{s}_{k}}}\]
  and summation operator is Schur convex (see e.g. \cite{ZS}), it follows 
 \[ Z_{S^{n}}(s,x,y) \leq \zeta_{g}(s,x,y). \]
 Therefore, 
 \[\sum_{k=1}^{\infty}  \frac{d_{k}(n) }{(k + \frac{n - 1}{2})^{ 2s}}    
\preceq \sum_{k=1}^{\infty}  \frac{d_{k}(n) }{(k(k+ n -1))^{s}}  ~~ \textrm{for}~~ \Re (s)>\frac{n}{2} .\]
\end{proof}

\section{Conclusion}
We have constructed Kato's bounds for the spectral zeta kernel of 
a  Schr$\ddot{o}$dinger-type operator in terms of the spectral zeta kernel of the 
Laplacian and Riemann zeta function on closed Riemannian manifolds.
 We proved that  $\tr\exp (-t H) \preceq \tr\exp (-t \Delta) $
for $H = \Delta _{g} + V$ on  smooth functions  on  $(M,g).$ 
Several illustrations were done on the $n$-dimensional unit sphere. A similar study can be 
 done on other Riemannian manifolds of higher genus and with boundary.

\section{Competing interests}
The authors declare that they have no competing interests.

\end{document}